\documentclass[12pt,reqno]{amsart}

\usepackage{amssymb}
\usepackage{amscd}
\usepackage{amsfonts}
\usepackage{setspace}
\usepackage{version}
\usepackage{mathrsfs}

\usepackage{graphicx}

\newtheorem{theorem}{Theorem}[section]
\newtheorem{lemma}[theorem]{Lemma}
\newtheorem{proposition}[theorem]{Proposition}
\newtheorem{corollary}[theorem]{Corollary}

\theoremstyle{definition}

\renewcommand{\leq}{\leqslant}
\renewcommand{\geq}{\geqslant}

\newcommand\triv{\operatorname{triv}}
\newcommand\sieve{\operatorname{sieve}}
\newcommand\bilinear{\operatorname{bilinear}}

\def\R{\mathbb{R}}
\def\C{\mathbf{C}}
\def\Z{\mathbf{Z}}
\def\E{\mathbf{E}}
\def\P{\mathbf{P}}

\def\N{\mathbf{N}}
\def\X{\mathsf{X}}

\def\eps{\varepsilon}

\newcommand{\md}[1]{\ensuremath{(\operatorname{mod}\, #1)}}
\newcommand{\mdsub}[1]{\ensuremath{(\mbox{\scriptsize mod}\, #1)}}

\newcommand{\mdsublem}[1]{\ensuremath{(\mbox{\scriptsize \textup{mod}}\, #1)}}

\numberwithin{equation}{section}

\begin{document}

\title[Multiplicative functions in progressions to large moduli]{A note on multiplicative functions on progressions to large moduli}


\author{Ben Green}
\address{Mathematical Institute, Radcliffe Observatory Quarter, Woodstock Rd, Oxford OX2 6GG}
\thanks{The author is supported by a grant from the Simons Foundation, and by ERC Advanced Grant AAS 279438. He thanks both organisations for their support.}

\begin{abstract}
Let $f : \N \rightarrow \C$ be a bounded multiplicative function. Let $a$ be a fixed nonzero integer (say $a = 1$). Then $f$ is well-distributed on the progression $n \equiv a \md{q} \subset \{1,\dots, X\}$, for almost all primes $q \in [Q,2Q]$, for $Q$ as large as $X^{\frac{1}{2} + \frac{1}{78} - o(1)}$. 
\end{abstract}
\maketitle

\section{Introduction}

Let $f : \N \rightarrow \C$ be a multiplicative function with $|f(n)| \leq 1$ for all $n$. In this note we look at how $f(n)$ behaves on progressions $n \equiv a \md{q}$, $n \leq X$, with $q$ a prime larger than $X^{1/2}$ by a small power.\vspace{8pt}

\emph{Notation.} Throughout the paper, $\mathbb{E}_{n \in S}$ is shorthand for $\frac{1}{|S|}\sum_{n \in S}$, where $S$ is a set of integers. We reserve the notation $\E$ for the expectation of a random variable. The letter $c$ denotes a positive absolute constant, which may be different at each appearance. When we write $X \ll Y$ we will mean that $|X| \leq CY$ for some absolute constant $C$, which may again be different at each appearance. We write $e(t)$ as a shorthand for $e^{2\pi i t}$.

\begin{theorem}\label{thm2.3}
Suppose that $f : \N \rightarrow \C$ is a multiplicative function with $|f(n)| \leq 1$ for all $n$. Suppose that $Q$ satisfies $X^{1/3} < Q < X^{\frac{1}{2} + \frac{1}{78} - \sigma}$, and suppose that $0 < |a| < 10Q$. Then \[ |\mathbb{E}_{n \leq X, n \equiv a \mdsublem{q}} f(n)  - \mathbb{E}_{n \leq X} f(n)|\leq \eps \] with the possible exception of at most $Q\eps^{-1}X^{-c\sigma\eps}$ primes $q$ with $Q \leq q < 2Q$.\end{theorem}

\emph{Remarks.} The parameter $\sigma$ is of very little consequence and is included just so that we can state the largest range of $Q$ for which we can get a nontrivial result. The range of $a$ stated is not the best one that can be obtained with our method, but the result is probably most interesting for constant $a$ (for example $a = 1$). The statement has content for $\eps \gg \frac{\log \log X}{\log X}$ and is perhaps most interesting for $\eps \gg 1$, in which case we get a power saving over the trivial bound. The stipulation that $Q > X^{1/3}$ is somewhat arbitrary: the main interest of the result is for $Q > X^{1/2 - o(1)}$.

Let $\mu$ be the M\"obius function and $\lambda$ the Liouville function. A straightforward corollary of Theorem \ref{thm2.3} and the classical estimates $\mathbb{E}_{n \leq X} \mu(n), \mathbb{E}_{n\leq X} \lambda(n) \ll e^{-\sqrt{\log X}}$ (both essentially the prime number theorem with classical error term) is the following.

\begin{corollary}
Suppose that $Q$ satisfies $X^{1/3} < Q < X^{\frac{1}{2} + \frac{1}{78} - \sigma}$, and suppose that $0 < |a| < 10Q$. Then $|\mathbb{E}_{n \leq X, n \equiv a \mdsublem{q}} \mu(n)|\leq \eps$ with the possible exception of at most $Q\eps^{-1}X^{-c\sigma \eps}$ primes $q$ with $Q \leq q < 2Q$. The same is true for the Liouville function $\lambda$.
\end{corollary}

\emph{Further remarks.} The main novelty in these results is that they apply with \emph{prime} moduli larger than $X^{1/2}$ by a power. There is a considerable and deep literature concerning the distribution of primes in progressions $a \md{q}$ with $q > X^{1/2}$. However, these works typically require $q$ to be ``smooth'' or ``well-factorable'' \cite{bfi1,fouvry-iwaniec,zhang} or else ``only'' beat the $X^{1/2}$ barrier by a smaller term $X^{o(1)}$ \cite{bfi2,bfi3}.

In fact, the restriction to prime $q$ (or at least some restriction on $q$) is quite important for us. Indeed, as stated, Theorem \ref{thm2.3} is false without some such restriction, as may be seen by taking $f(n) = (-1)^n \mu^2(n)$. In this case, $\mathbb{E}_{n \leq X} f(n)$ is very small, but $|\E_{n \leq X, n = a \md{q}} f(n)| \gg 1$ whenever $q$ is even. The issue here is that $f$ is ``pretentious'' in the sense of Granville and Soundararajan (see, for example, \cite{bgs}). We do expect Theorem \ref{thm2.3} to hold when $f = \mu$ is the M\"obius function, even if $q$ is allowed to be composite. However, our methods do not give such a statement in their current form.

Very little in the proof of our main result will be a surprise to experts. Many of the ingredients (application of Poisson summation formula, reciprocity for congruences, reduction to a bilinear form estimate) may be found in the literature cited above. 

The main difference between our work and the aforementioned papers is that, because our interest is in bounded multiplicative functions such as $\mu$ rather than in the primes, we can make do with bilinear form estimates in a rather restricted (and accessible) range. This observation goes back to K\'atai \cite{katai} and has featured, in a variety of different forms, in many recent works. In particular we mention the paper of Bourgain-Sarnak-Ziegler \cite{bgs}, as well as the note of Harper \cite{harper}. 

To obtain bilinear forms we will proceed using an identity of Ramar\'e, which affords some flexibility in the choice of various parameters and leads to quite good bounds. We do not claim any originality for the idea of using Ramar\'e's identity in this context: it is implicit in \cite{Matomaki-Radziwill}, and remarked upon without proof in a blog post of Tao \cite[Remark 4]{tao-blog}. However, we do not know of any portable implementation of this in the literature, and so Proposition \ref{bilinear} may be useful elsewhere.

To analyse the resulting bilinear forms, which involve ``Kloosterman fractions'', we use a result of Duke, Friedlander and Iwaniec \cite{dfi} (we in fact use a numerically stronger version of this due to Bettin and Chandee \cite{bettin-chandee}). 

In Theorem \ref{thm2.3}, the modulus $a$ was fixed. We do not know how to establish a corresponding result with $a$ being allowed to vary with $q$, for any $Q > X^{1/2}$. This in fact appears to be a significantly harder problem, leading to issues related to the Kakeya conjecture in Euclidean harmonic analysis. We make some comments on this point in Section \ref{var-cong-sec}. In a forthcoming joint paper with I.~Z.~Ruzsa \cite{green-ruzsa}, we explore related issues in more detail.

\emph{Acknowledgement.} I thank Fernando Shao for suggesting I look at the case of a fixed residue class $a$. I am also grateful to Guangshi Lv for pointing out a small technical mistake in the first version of the paper.

\section{Ramar\'e's weights and bilinear forms}

In this section we record a result, Proposition \ref{bilinear} below, which is of a type well-known to experts. It is very similar to results of Katai \cite{katai}, Bourgain-Sarnak-Ziegler \cite{bsz} and Harper \cite{harper}, and is based on work of Ramar\'e \cite{ramare}. We thank Kaisa Matom\"aki and Terence Tao for drawing our attention to Ramar\'e's work, especially the latter who provided some details of it in his blog.

 Throughout this section, we will have parameters $Y < Z$, depending on $X$. Write $u := \frac{\log Z}{\log Y}$, and assume that $u \geq 2$.  We will consider the weight function
\begin{equation}\label{ramare-def} w(n) := \frac{1}{\# \{p : Y \leq p < Z : p | n \} + 1}.\end{equation}
\begin{lemma}\label{lem2.1}
Suppose $M \geq Z^8$. Then $\mathbb{E}_{M \leq m < 2M} w(m)^2 \ll  (\log u)^{-2}$.
\end{lemma}
\begin{proof} We may assume that $u$ is sufficiently large, the result being vacuous otherwise.
Suppose that $m$ is selected at random from $[M, 2M)$. Write $\X_p$ for the event that $p | m$, and write $\X = \sum_{Y \leq p < Z} \X_p$; then the quantity we wish to bound is $\E \frac{1}{(\X + 1)^2}$.

We have
\[ \E \X_p = \frac{1}{p} + O(\frac{1}{M})\]
and so
\[ \E \X = \sum_{Y \leq p < Z} \frac{1}{p} + O(\frac{Z}{M}) = \log u + O(1)\] by Mertens' estimate. We claim that it is enough to show
\begin{equation}\label{enough-3} \P(\X = t) \ll (\log u)^{-2}\end{equation} uniformly for $0 \leq t \leq \frac{1}{2}\log u$. Indeed, we then have
\begin{align*} \E \frac{1}{(\X + 1)^2} & \ll (\log u)^{-2} + \sum_{0 \leq t \leq \frac{1}{2} \log u} \frac{\P(\X = t)}{(t+1)^2} \\ & \ll (\log u)^{-2} \big(1 + \sum_{0 \leq t \leq \frac{1}{2}\log u} \frac{1}{(t+1)^2} \big) \ll (\log u)^{-2},\end{align*} which is what we wanted to prove.
To establish \eqref{enough-3}, we use a fourth moment argument. Write $\X'_p := \X_p - \frac{1}{p}$ and $\X' := \sum_{Y \leq p < Z} \X'_p = \X - \log u + O(1)$. If $\X = t$ with $t \leq \frac{1}{2} \log u$ then $\X' > \frac{1}{4}\log u$, so 
\[ \P(\X = t) \leq \P(\X' > \frac{1}{4} \log u) \ll (\log u)^{-4} \E \X^{\prime 4}\] by Markov's inequality. It therefore suffices to prove that
\begin{equation}\label{eq47} \E \X^{\prime 4} \ll (\log u)^2.   \end{equation} 
 To this end, we may expand
\begin{equation}\label{eq48} \E \X^{\prime 4}  =  \sum_{Y \leq p_1,p_2,p_3,p_4 < Z}\E \X'_{p_1} \X'_{p_2} \X'_{p_3} \X'_{p_4}.\end{equation}
Suppose that one of the $p_i$ is different from all the others. Then, by the Chinese remainder theorem, the sum of $\prod_{i = 1}^4 (1_{p_i | m} - \frac{1}{p_i})$ over any interval of length $p_1p_2p_3p_4$ is $0$, and so
\[ \E \X'_{p_1} \X'_{p_2} \X'_{p_3} \X'_{p_4} = O(\frac{Z^4}{M}).\] 
The contribution to \eqref{eq48} from these choices is therefore $\ll Z^8/M \ll 1$. 
If there is no such $p_i$ then the average is of the form $\E \X_p^{\prime 2} \X_{p'}^{\prime 2}$. Noting that $\X_p^j = \X_p$ for all $j$ (since $\X_p$ takes values $0$ and $1$) and that $\E \X_p \X_{p'} = \frac{1}{pp'} + O(\frac{1}{M}) \ll \frac{1}{pp'}$ if $p \neq p'$, it follows easily from the binomial theorem that 
\[ \E \X_p^{\prime 2} \X_{p'}^{\prime 2} \ll \left\{\begin{array}{ll}   1/pp' & \mbox{if $p \neq p'$} \\ 1/p & \mbox{if $p = p'$}.\end{array} \right. \] 
The contribution to \eqref{eq48} from these remaining quadruples $p_1,p_2,p_3,p_4$ is therefore
\[ \sum_{Y \leq p, p' < Z, p \neq p'} \frac{1}{p p'} + \sum_{Y \leq p < Z}  \frac{1}{p} \ll (\log u)^2.\]
This concludes the proof of \eqref{eq47} and hence of the lemma.
\end{proof}

\begin{proposition}\label{bilinear}
Let $F : \N \rightarrow \C$ be any function. Let $1 < Y < Z < X^{1/16}$ be parameters. Let $f : \N \rightarrow \C$ be a multiplicative function with $|f(n)| \leq 1$ for all $n$. Then 
\[  \mathbb{E}_{n \leq X} f(n) F(n)  \ll E_{\triv} + E_{\sieve} + E_{\bilinear},\] where
\[ E_{\triv} := Y^{-1/2}\Vert F \Vert_{\infty} ,\]
\[ E_{\sieve} := \mathbb{E}_{n \leq X}|F(n)|1_{(n,\prod_{Y \leq p < Z} p) = 1},\]
and 
\[ E_{\bilinear} = \sqrt{\sup_{p,p', I} \frac{1}{\max(I)}|\sum_{\substack{m \in I\\ (m,pp') = 1}} F(pm) \overline{F(p' m)}|},\] the supremum being taken over all primes $p, p'$ with $Y \leq p < p' < Z$ and all intervals $I \subset [0,X]$ with $\max(I) > X/10YZ$.
 \end{proposition}
\emph{Remarks.} So as to get a fairly clean statement, we have omitted some logarithmic factors which, if included, would make the statement marginally stronger. 

One could formulate and prove, using an almost identical argument, a similar statement in which some small collection of exceptional pairs of primes $p, p'$ was tolerated in the definition of $E_{\bilinear}$. We do not do this here, since it is not necessary for our applications. One problem where such a formulation could be of interest is the (open) question of showing that $\sum_{n \leq X} \mu(n) \Lambda( n - 1) = o(X)$. Here, the bilinear estimate is certainly the heart of the matter. Things would be reduced to showing that for almost every pair of primes $p,p' \sim Q$ with $Q \sim X^{\delta}$ one has the expected number (as predicted by the Hardy-Littlewood heuristics) of $x \lessapprox X/Q$ for which $px + 1, p'x + 1$ are both prime. 

\begin{proof}
Set $u := \frac{\log Z}{\log Y}$.
Let $w$ be Ramar\'e's weight function as defined in \eqref{ramare-def} above. It is convenient to introduce the function $\mu^2_{[Y,Z)}(n)$, defined to be $0$ if $n$ is divisible by the square of some prime $p$ with $Y \leq p < Z$, and $1$ otherwise. If $\mu^2_{[Y,Z)}(n) = 1$ then we have the (easily-checked) Ramar\'e identity
\[ \sum_{\substack{Y \leq p < Z\\ p | n}} w(\frac{n}{p}) = \left\{ \begin{array}{ll} 1 & \mbox{if $p | n$ for some $Y \leq p < Z$} \\ 0 & \mbox{otherwise}.   \end{array}\right.\]
Now we have
\begin{align*} \sum_{\substack{n \leq X \\ \mu^2_{[Y,Z)}(n) = 0}} f(n) F(n) \leq \sum_{Y \leq p < Z} \sum_{\substack{n \leq X \\ p^2 | n}} |f(n) F(n)| \leq \Vert F \Vert_{\infty} \sum_{Y \leq p < Z} \sum_{\substack{n \leq X \\ p^2 | n}} 1 \\  \leq \Vert F \Vert_{\infty} \sum_{Y \leq p < Z} (\frac{X}{p^2} + O(1)) \ll \Vert F \Vert_{\infty} (Z + \frac{X}{Y}) \ll X E_{\triv}.\end{align*}
Meanwhile, from the Ramar\'e identity, 

\begin{align*} \sum_{\substack{n \leq X \\ \mu^2_{[Y,Z)}(n) = 1 }} & f(n) F(n) \\ & = \sum_{\substack{n \leq X \\ \mu^2_{[Y,Z)}(n) = 1 }} f(n) F(n) \sum_{\substack{Y \leq p < Z \\ p | n}} w(\frac{n}{p}) + \sum_{\substack{ n \leq X \\  (n,\prod_{Y \leq p < Z}p) = 1 }}  \!\!\!\! f(n) F(n). \\ & =  \sum_{\substack{n \leq X \\ \mu^2_{[Y,Z)}(n) = 1 }} f(n) F(n) \sum_{\substack{Y \leq p < Z \\ p | n}} w(\frac{n}{p}) + \;\;\;\;O(XE_{\sieve}).\end{align*} The main business of the proof is therefore to bound the first term, which we refer to as $\Sigma$ from now on, by $O(X)(E_{\triv} + E_{\sieve} + E_{\bilinear})$.

Using the multiplicativity of $f$, $\Sigma$ may be rewritten as
\begin{align*} \sum_{Y \leq p < Z} f(p)\sum_{\substack{m \leq X/p \\ (m,p) = 1 \\ \mu^2_{[Y,Z)}(m) = 1}} w(m) f(m) F(pm) \\ = \sum_{\substack{m \leq X/Y \\ \mu^2_{[Y,Z)}(m) = 1}} w(m) f(m) \sum_{\substack{Y \leq p < Z \\ p \leq X/m \\ (m,p) = 1}} f(p) F(pm).\end{align*}

Since $w(m) \leq 1$ and $|f(m)| \leq 1$ pointwise, a very crude estimate for the contribution to $\Sigma$ from $m \leq X^{1/2}$ is $O(X^{1/2} Z \Vert F \Vert_{\infty}) \ll X E_{\triv}$. 
Split the sum over the remaining $m$ into exponential ranges $e^{-i-1}X < m \leq e^{-i}X$, where $\log Y \leq i \leq \frac{1}{2}\log X$. For notational convenience we write this range as $m \sim e^{-i}X$. On each such range we may apply the Cauchy-Schwarz inequality, obtaining a bound

\begin{align}\nonumber  \Sigma'  & \ll X E_{\triv} + \\ \nonumber &  \sum_{i =\log Y }^{\frac{1}{2}\log X} \bigg( \sum_{\substack{m \sim e^{-i}X \\ \mu^2_{[Y,Z)}(m) = 1}} w(m)^2 \bigg)^{1/2} \bigg( \sum_{\substack{m \sim e^{-i}X \\ \mu^2_{(Y,Z]}(m) = 1}} \big| \sum_{\substack{Y \leq p < Z \\ p \leq X/m \\ (m,p) = 1}} f(p) F(pm)\big|^2 \bigg)^{1/2} \\ \label{eq98a}& \ll X E_{\triv} + \frac{X^{1/2}}{\log u} \sum_{i = \log Y}^{\frac{1}{2}\log X} e^{-i/2} \bigg( \sum_{m \sim e^{-i}X} \big| \sum_{\substack{Y \leq p < Z \\ p \leq X/m \\ (m,p) = 1}} f(p) F(pm)\big|^2 \bigg)^{1/2}.\end{align} In deriving the second line here we made use of Lemma \ref{lem2.1}; this is valid since, with $i$ in the stated range, $e^{-i} X \geq X^{1/2} \geq Z^8$. Observe that positivity has allowed us to drop the condition $\mu^2_{[Y,Z)}(m) = 1$.

We have
\begin{align}\nonumber \sum_{m \sim e^{-i}X} \big| & \sum_{\substack{Y \leq p < Z \\ p \leq X/m}} f(p) F(pm)\big|^2 \\ \nonumber & = \!\!\!\! \sum_{Y \leq p, p' < \min(e^i, Z)} \!\!\!\!\! f(p) \overline{f(p')}\!\!\!\!\!\! \sum_{\substack{m \sim e^{-i}X \\ m \leq \min(X/p, X/p') \\ (m, pp') = 1}} \!\!\!\!\!\! F(pm) \overline{F(p' m)} \\ \label{eq90a} & \leq \sum_{Y \leq p, p' < \min(e^i, Z)} \big| \!\!\!\!\! \sum_{\substack{m \sim e^{-i}X \\ m \leq \min(X/p, X/p') \\ (m,pp') = 1}}\!\!\!\! F(pm) \overline{F(p' m)}\big| .\end{align}
Since $(a + b)^{1/2} \ll a^{1/2} + b^{1/2}$ for $a, b > 0$, we have, comparing with \eqref{eq98a},
\begin{align*} \Sigma & \ll X E_{\triv} + \frac{X^{1/2}}{\log u} \sum_{i = \log Y}^{\frac{1}{2}\log X} e^{-i/2}  \bigg( \sum_{Y \leq p \leq \min(e^i, Z)} \sum_{\substack{m \sim e^{-i} X \\ m \leq X/p \\ (m,p) = 1}} |F(pm)|^2\bigg)^{1/2} + \\ & \quad \frac{X^{1/2}}{\log u} \sum_{i = \log Y}^{\frac{1}{2}\log X} e^{-i/2}\bigg(\sum_{Y \leq p < p' < \min(e^i, Z)} \big| \!\!\!\!\! \sum_{\substack{m \sim e^{-i}X \\ m \leq \min(X/p, X/p') \\ (m,pp') = 1}} F(pm) \overline{F(p' m)}\big|\bigg)^{1/2} \\ & = X E_{\triv} + E_1 + E_2,\end{align*}
say.

Let us first estimate $E_1$. Rather crudely,
\begin{align*}  \sum_{Y \leq p \leq \min(e^i, Z)} \sum_{\substack{m \sim e^{-i}X \\ m \leq X/p \\ (m,p) = 1}} |F(pm)|^2 & \ll X\Vert F \Vert_{\infty}^2, \end{align*} and so
\[ E_1 \leq \frac{X\Vert F \Vert_{\infty}}{\log u} \sum_{i = \log Y}^{\frac{1}{2}\log X} e^{-i/2} \ll \frac{X}{Y^{1/2}} \Vert F \Vert_{\infty} = XE_{\triv}.\]
(We simply ignored the $\log u$ in the denominator, which will be very small compared to $Y^{1/2}$ in applications.) 
 
Next we bound $E_2$. For the portion of the sum with $i > \log(YZ)$ we use the trivial bound
\[ \frac{X^{1/2}}{\log u} \sum_{i : e^i > YZ} e^{-i/2}\big( Z^2 \cdot (e^{-i} X \Vert F \Vert_{\infty}^2 \big)^{1/2}  \ll  \frac{X\Vert F \Vert_{\infty}}{Y\log u}.\] This is bounded by $XE_{\triv}$.

Recalling the definition of $E_{\bilinear}$, for $e^i < YZ$ we have
\begin{align*} \sum_{Y \leq p < p' < \min(e^i, Z)} & \big| \!\!\!\!\sum_{\substack{m \sim e^{-i}X \\ m \leq \min(X/p, X/p')\\ (m,pp') = 1}} \!\!\!\!\!\!F(pm) \overline{F(p' m)}\big| \\ & \ll  E_{\bilinear}^2 e^{-i} X \# \{ Y < p < p' < \min(e^i, Z)\}  \\ & \ll 
 \left\{\begin{array}{ll} E_{\bilinear}^2\frac{ e^{i} X}{i^2} & \mbox{if } i \leq \log Z \\ E_{\bilinear}^2 e^{-i} X \big( \frac{Z}{\log Z}\big)^2& \mbox{otherwise}. \end{array}   \right.\end{align*}It follows that the remaining portion of $E_2$ (that is, the sum over $i < \log (YZ)$) is bounded by $E_{\bilinear}$ times
\[ \frac{X}{\log u} \sum_{i = \log Y}^{\log Z}  \frac{1}{i} + \frac{XZ}{\log u\log Z} \sum_{i = \log Z}^{\log (YZ)} e^{-i}  \ll  X.
\]
Here, we noted that 
\[ \sum_{i = \log Y}^{\log Z} \frac{1}{i} = \log \big( \frac{\log Z}{\log Y} \big) + O(1) = \log u + O(1).\]

Putting all this together concludes the proof.
\end{proof}

\section{Proof of the main theorem}

We will prove the following statement, which implies the main theorem in a very straightforward manner.

\begin{proposition}\label{prop3.1}
Suppose that $X^{1/3} < Q < X^{\frac{1}{2} + \frac{1}{78} - \sigma}$. Let $-10Q < a < 10Q$ and let $F : \N \rightarrow \C$ be any function of the following form:
\[ F(n) =  \left\{\begin{array}{ll} \sum_{Q \leq q < 2Q} \xi_q (1_{n \equiv a \mdsublem{q}} - \frac{1}{q}) & n \neq a \\ 0 & n = a\end{array}\right.\] if $n \neq a$, and $F(a) = 0$, where $(\xi_q)_{Q \leq q < 2Q}$ is a sequence of complex numbers satisfying $|\xi_q| \leq 1$ for all $q$ and $\xi_q = 0$ unless $q$ is prime. Let $f : \N \rightarrow \C$ be a multiplicative function with $|f(n)| \leq 1$ for all $n$. Let $\eps > 0$. Then 
\[ \mathbb{E}_{n \leq X} f(n) F(n) \ll  \frac{\eps}{Q}\sum_{Q \leq q < 2Q} |\xi_q| + X^{-\sigma\eps/20}.\] 
\end{proposition}

Let us see how our main theorem, Theorem \ref{thm2.3}, follows from this. The remainder of the paper will then be devoted to the proof of Proposition \ref{prop3.1}.\vspace{8pt}

\emph{Deduction of Theorem \ref{thm2.3}.} Suppose that there is a set $S \subset [Q,2Q]$ of primes such that 
\[ |\mathbb{E}_{x \leq X, x \equiv a \mdsub{q}} f(x) - \mathbb{E}_{x \leq X} f(x)| \geq \eps\] for all $q \in S$. For $q \in S$, choose unit-modulus complex numbers $\xi_q$ such that 
\[ \xi_q \big(\mathbb{E}_{x \leq X, x \equiv a \mdsub{q}} f(x) - \mathbb{E}_{x \leq X} f(x)\big) \geq \eps.\] For $q \notin S$, set $\xi_q = 0$. Then, taking $F(n)$ as in the statement of Proposition \ref{prop3.1} (with this choice of $\xi_q$) we have
\begin{equation}\label{to-comp} \mathbb{E}_{n \leq X} f(n) F(n) \geq \frac{\eps }{2} \sum_{q \in S} \frac{1}{q} \geq \frac{\eps}{4} \frac{\# S}{Q}.\end{equation} Here, we used the fact that $q/X$ is much smaller than $\eps$; we may certainly assume this since Proposition \ref{prop3.1} has no content when $\eps < \frac{1}{\log X}$.

However, Proposition \ref{prop3.1} provides the upper bound
\[ \mathbb{E}_{n \leq X} f(n) F(n) \ll \frac{\eps'}{Q}\# S + X^{-\sigma \eps' /20}.\]
Taking $\eps' = c\eps$ for a suitably small constant $c$, it follows that 
\[ X^{-\sigma \eps'/20} \gg \frac{\eps}{Q} \# S.\]
This concludes the deduction of Proposition \ref{prop3.1} from Theorem \ref{thm2.3}. \vspace{11pt}

\emph{Proof of Proposition \ref{prop3.1}.} We apply Proposition \ref{bilinear}, taking $Y = X^{\eps\sigma/4}$ and $Z = X^{\sigma/4}$.  \vspace{8pt}

\emph{Estimation of $E_{\triv}$.} Note that $\Vert F \Vert_{\infty} \leq 2$, since all the primes $q$ are $> X^{1/3}$ and so if $n \equiv a \md{q}$ for at least $3$ different $q$ then $n = a$. Thus the contribution of $E_{\triv}$ is $\ll Y^{-1/2} = X^{-\eps\sigma/8}$, which is one of the terms in the statement of Proposition \ref{prop3.1}.\vspace{8pt}

\emph{Estimation of $E_{\sieve}$.} We use the fact that 
\begin{align} \nonumber \# \{ n \leq X : n \equiv a \md{q}, (n, \prod_{Y \leq p < Z} p) = 1\}  & \ll\frac{X}{q} \prod_{Y \leq p < Z} (1 - \frac{1}{p}) \\ & \ll  \frac{\log Y}{\log Z} \frac{X}{q},\label{sieve}\end{align} uniformly for $1 < Y < Z < X^{1/10}$, for $q < X^{3/4}$ and for all $a$. Such an estimate is a consequence of the fundamental lemma of the combinatorial sieve.

By the triangle inequality
\begin{align*}
XE_{\sieve} & = \sum_{n \leq X} |F(n)| 1_{(n, \prod_{Y \leq p < Z} p) = 1} \\ & \leq \sum_{Q \leq q < 2Q} |\xi_q|\#\{n \leq X : n \equiv a \md{q}, (n, \prod_{Y \leq p < Z} p) = 1\}  \\ & \qquad + \sum_{Q \leq q < 2Q} \frac{|\xi_q|}{q} \#\{n \leq X : (n, \prod_{Y \leq p < Z} p) = 1\}.
\end{align*}
By \eqref{sieve}, this is bounded by 
\[ \frac{X\log Y}{\log Z}\sum_{Q \leq q < 2Q} \frac{|\xi_q|}{q}.\]
With our choice of $Y, Z$, it follows that 
\[ E_{\sieve} \leq \frac{\eps}{Q} \sum_{Q \leq q < 2Q} |\xi_q|.\]
The right-hand side is one of the terms in the statement of Proposition \ref{prop3.1}.\vspace{8pt}

\emph{Estimation of $E_{\bilinear}$.} This is the heart of the matter. It is enough to show that 
\begin{equation}\label{bil-est} \sum_{\substack{m \in I \\ (m,pp') = 1}} F(pm) F(p' m) \ll L X^{-\eps\sigma/10}\end{equation} whenever $I$ is a subinterval of $[0, L]$, whenever $p \neq p'$ are distinct primes with $Y \leq p < p' < Z$, and for all $L$ with $X^{1 - \sigma/2} \leq L \leq X$ (in the notation of Proposition \ref{bilinear}, $L = e^{-i} X$, and we actually need the estimate for $L \gg X/YZ$; however, $Y \ll Z = X^{\sigma/4}$.).

It is convenient to remove the condition $(m,pp') = 1$. The contribution to the left-hand side of \eqref{bil-est} from $m$ not satisfying this condition is $\ll L/Y$, which is certainly acceptable. 

For the remaining sum we will in fact show the stronger estimate
\begin{equation}\label{bil-est-2} \sum_m 1_J(\frac{m}{L}) F(pm) F(p'm) \ll LX^{-\sigma/10},\end{equation} where $J \subset [0,1]$ is a subinterval of $\R$.

The next step, completely routine in considerations of this type, is to replace the cutoff $1_J$ by a smooth variant. Set
\[ W(x) = \int 1_J(y) X^{\sigma/10}\Psi(X^{\sigma/10} (x -y))dy,\] where $\Psi \in C_0^{\infty}(\R)$ has $\Psi \geq 0$, $\mbox{Supp}(\Psi) \subset [-1,1]$, $\int \Psi = 1$. Then $W = 1_J$ outside a union of two intervals of measure $O(X^{-\sigma/10})$, and so it suffices to show that 
\begin{equation}\label{smoothed-sum} \sum_m W(\frac{m}{L}) F(pm) F(p'm) \ll LX^{- \sigma/10}.\end{equation}
Let $A \geq 2$ be an integer. Noting that
\[ \Vert W^{(A)}\Vert_{\infty} \leq X^{A\sigma/10} \Vert \Psi^{(A)} \Vert_1 \ll_A X^{A \sigma/10} \] and that $W$ is constant outside of the union of two intervals of measure $O(X^{-\sigma/10})$, we have the derivative bound $\Vert W^{(A)} \Vert_1 \ll X^{(A - 1)\sigma/10}$. Therefore by partial integration we have the Fourier bound
\begin{equation}\label{fourier-w} |\hat{W}(\xi)| \ll_A |\xi|^{-A}
\Vert W^{(A)} \Vert_1  \ll_A |\xi|^{-A} X^{(A - 1)\sigma/10}.\end{equation}
To proceed further towards \eqref{smoothed-sum} we expand out the definition of $F$, reducing the task to proving 
\begin{align*} \sum_m W(\frac{m}{L}) \sum_{Q \leq q,q' \leq 2Q} \xi_q \xi_{q'}(1_{pm \equiv a \mdsub{q}} - \frac{1}{q})&(1_{p'm \equiv a\mdsub{q'}} - \frac{1}{q'}) \\ & \ll LX^{ - \sigma/10}\end{align*} for any choice of $\xi_q$, $|\xi_q| \leq 1$, $\xi_q = 0$ unless $q$ is prime. Using the identity $(a - a')(b - b') = -(ab - a'b') + a(b - b') + b(a - a')$ this may be further split into the following subtasks:

\begin{equation}\label{subtask-1} \sum_{q'} \frac{\xi_{q'}}{q'} \sum_m W(\frac{m}{L})\sum_q \xi_q (1_{pm \equiv a \mdsub{q}} - \frac{1}{q}) \ll LX^{- \sigma/10} \end{equation}
\begin{equation}\label{subtask-2} \sum_{q} \frac{\xi_{q}}{q} \sum_m W(\frac{m}{L})\sum_{q'} \xi_{q'} (1_{p'm \equiv a \mdsub{q'}} - \frac{1}{q'}) \ll LX^{ - \sigma/10} \end{equation} and
\begin{equation}\label{subtask-3}  \sum_{q,q'}\sum_m W(\frac{m}{L}) \xi_q\xi_{q'} (1_{pm \equiv a \mdsub{q}}1_{p'm \equiv a \mdsub{q'}} - \frac{1}{qq'}) \ll LX^{- \sigma/10} \end{equation}
Of these, \eqref{subtask-1} and \eqref{subtask-2} are equivalent and so we need only prove one of them, say \eqref{subtask-1}; since $\sum_{q'} \frac{\xi_{q'}}{q'} = O(1)$, it is enough to prove that
\begin{equation}\label{subtask-4}  \sum_q | \sum_m W(\frac{m}{L}) (1_{pm \equiv a \mdsub{q}} - \frac{1}{q})| \ll LX^{- \sigma/10}. \end{equation}
Thus \eqref{subtask-3} and \eqref{subtask-4} are our remaining tasks. The first step in establishing both of them is an application of the Poisson summation formula. In the case of \eqref{subtask-4}, this is essentially also the last step. By contrast, \eqref{subtask-3} lies deeper.

The Poisson summation formula $\sum_{n \in \Z} \phi(n) = \sum_{h \in \Z} \hat{\phi}(2\pi h)$ applied with $\phi(x) = W(\frac{dx + b}{L})$ gives
\begin{equation}\label{pois} \sum_m W(\frac{m}{L})(1_{m \equiv b \md{d}} - \frac{1}{d} ) = \frac{L}{d} \sum_{h \neq 0} \hat{W}(\frac{2\pi Lh}{d}) e(\frac{bh}{d})\end{equation}
To prove \eqref{subtask-4}, we can proceed with rather crude bounds: using \eqref{fourier-w} with $A = 2$ we have
\begin{align*}
\sum_q | \sum_m W(\frac{m}{L}) (1_{pm \equiv a \mdsub{q}} - \frac{1}{q})|   \leq \sum_q \frac{L}{q} \sum_{h \neq 0} |\hat{W}(\frac{2\pi Lh}{q})| \\ \ll \sum_q \frac{L}{q} \sum_{h \neq 0} X^{\kappa} |\frac{L h}{q}|^{-2} \ll Q^2 X^{\sigma/10}L^{ - 1} \ll L X^{-\sigma/10}
\end{align*}
provided that $Q < X^{1 - \sigma}$.
This establishes \eqref{subtask-4}. Turning to \eqref{subtask-3}, a similarly blunt approach would lead only to a corresponding bound under the much stronger condition $Q^2 \leq X^{1 - O(\sigma)}$, which excludes any possibility of working with $Q > X^{1/2}$. To access this range we must exploit cancellation coming from the phases $e(\frac{bh}{d})$ in \eqref{pois}.

Let us turn to the details (of bounding \eqref{subtask-3}). Let us first make the trivial observation that the contribution from $q = q'$ is negligible. For the remaining pairs $q \neq q'$, the Chinese remainder theorem of course tells us that there is a unique residue class $r(q,q') \in \Z/qq' \Z$ such that $p r(q,q') \equiv a \md{q}$, $p' r(q,q') \equiv a \md{q'}$. The task is then to show that 
\begin{equation}\label{subtask-4prime} \sum_{q \neq q'} \xi_q \xi_{q'}\sum_m W(\frac{m}{L}) (1_{m \equiv r(q,q') \mdsub{q q'}} - \frac{1}{q q'}) \ll LX^{ - \sigma/10}.\end{equation} By Poisson summation, this follows from 
\[\sum_{h \neq 0} \sum_{q \neq q'} \frac{\xi_q \xi_{q'}}{qq'}  \hat{W}(\frac{2\pi L h}{q q'}) e(\frac{r(q, q') h}{q q'}) \ll X^{- \sigma/10}.\]
We bound the contribution from ``large'' $h$ trivially using \eqref{fourier-w}:
\begin{align*}
 \sum_{|h| > H} \sum_{q \neq q'} \frac{\xi_q \xi_{q'}}{q q'} \hat{W}(\frac{2\pi L h}{q q'}) e(\frac{r(q, q') h}{q q'})  \ll \sum_{|h| > H} \sum_{q \neq q'} \frac{1}{q q'} |\hat{W}(\frac{2\pi Lh}{q q'})| \\  \ll \sum_{q \neq q'} \frac{1}{q q'} X^{(A - 1)\kappa} \sum_{|h| > H} |\frac{Lh}{q q'}|^{-A} \ll X^{(A - 1)\sigma/10} Q^{2A} L^{ - A} H^{1 - A}.
\end{align*}
If $Q = X^{\frac{1}{2} + \eta}$ then one may compute that, with the choice $A = \lceil 100/\sigma\rceil$ and $H = X^{2\eta + 2\sigma/5}$, this contribution is bounded by $L X^{-\sigma/10}$ as required. It is thus enough to show that
\begin{equation}\label{to-show-45} \sum_{q \neq q'} \frac{\xi_q \xi_{q'}}{q q'} \hat{W}(\frac{2\pi L h}{q q'}) e(\frac{r(q, q') h}{q q'}) \ll X^{- 2\eta - \sigma/2}.\end{equation}
uniformly in $0 < h \leq H$. (The reader should have in mind that $H \sim X^{2\eta} = Q^2/X$, for some rough sense of the symbol $\sim$. Note that there is nothing to prove if $\eta < 0$.)

We now use a devious separation of variables trick from \cite[p267]{opera-cribro}. By a change of variables in the definition of the Fourier transform $\hat{W}$, we have
\[ \hat{W}(\frac{2\pi L h}{qq'}) = q \int_{|u| \leq 10/q} W(qu) e(-\frac{Luh}{q'}) du,\] and so the left-hand side of \eqref{to-show-45} is equal to
\begin{equation}\label{to-show-46} Q \int_{|u| \leq 10/q} du \bigg(\frac{1}{Q^2}\sum_{\substack{Q \leq q < 2Q \\ q \neq q'}} \alpha_u(q) \beta_{u,h}(q') e(\frac{r(q,q')h}{q q'} )  \bigg),\end{equation} where
\[ \alpha_u(q) = \xi_q W(qu)\] and
\[ \beta_{u,h}(q') := \frac{Q\xi_{q'}}{q'} e(-\frac{Lu h}{q'}).\]
The scalars $\alpha_u(q), \beta_{u,h}(q')$ are essentially arbitrary bounded functions of $q,q'$. Thus we do indeed choose to forget their precise form, thereby reducing matters to establishing the bilinear form estimate
\begin{equation}\label{to-show-47} \sum_{\substack{Q \leq q < 2Q \\ q \neq q'}} \alpha(q) \beta(q') e_{qq'}(-r(q,q') h) \ll X^{-2\eta - \sigma/2}\end{equation} for all choices of $\alpha(q), \beta(q')$ with $|\alpha(q)| \leq 1$, $|\beta(q')| \leq 1$, and uniformly for $h < H = X^{2\eta + 2\sigma/5}$. Here, and below, we have written $e_m(x)$ as a shorthand for $e(\frac{x}{m}) = e^{2\pi i x/m}$.

To proceed further we must be more explicit about $r(q,q')$ which, recall, is the solution to the simultaneous congruences \[ p r(q, q') \equiv a \md{q},\] \[ p' r(q,q') \equiv a \md{q'}.\] Note that \[ r(q,q') = a(p q')^{-1} \md{q} q' + a(p' q)^{-1}\md{q'} q,\] 
and so 
\begin{align} \nonumber e_{q q'}(h r(q, q')) & = e_{q} (ah (p q')^{-1} \md{q}) e_{q'} (ah (p' q)^{-1} \md{q'}) \\ & = \label{star-14} e_{p'q} (ah p' (p q')^{-1} \md{q}) e_{pq'} (ah p (p' q)^{-1} \md{q'}). \end{align}

Now we note the ``reciprocity relation''

\[ \frac{v^{-1}\md{u}}{u}  + \frac{u^{-1}\md{v}}{v}  \equiv \frac{1}{uv} \md{1},\]
which means that 
\[ e_u(v^{-1}\md u) = e_v(-u^{-1}\md{v} )e^{2\pi i /uv}.\]
Applying with $u = pq'$ and $v = p'q$ gives
\[ e_{pq'}(ahp (p' q)^{-1} \md{pq'}) = e_{p'q}(-ah p (p q')^{-1} \md{p' q}) (1 + O(\frac{|ah|}{Q^2})),\] 
and so from \eqref{star-14} 
\[  e_{q q'}(-h r(q,q')) = e_{p' q}(ah (p - p') (p q')^{-1} \md{p' q}) (1 + O(\frac{|ah|}{Q^2})).\]
Since $|a| \ll Q$, $|h| \lll X^{1/6}$, the error term is negligible for the purposes of establishing \eqref{to-show-47}. Therefore we see that it is now enough to establish
\begin{equation}\label{enough-to-show-8}
\sum_{q \neq q'} \alpha(q) \beta(q') e_{p' q}(ah (p - p')(pq')^{-1}\md{p' q}) \ll X^{-2\eta - \sigma/2}.
\end{equation}
Writing $m = pq'$, $n = p' q$, $b = ah(p' - p)$, $\tilde\alpha(m) = \alpha(\frac{m}{p})$ when $m/p$ is a prime in $[Q,2Q]$ and $0$ otherwise, $\tilde\beta(n) = \beta(\frac{n}{p'})$ when $n/p'$ is a prime in $[Q,2Q]$ and 0 otherwise, this takes the form
\[ \Sigma := \sum_{\substack{pQ \leq m < 2pQ \\ p'Q \leq n < 2p' Q}} \tilde\alpha(m) \tilde\beta(n) e_n(b m^{-1} \md{n}).\] 
Nontrivial bounds for bilinear forms of this type were given by Duke, Friedlander and Iwaniec \cite{dfi}. A much more recent paper by Bettin and Chandee \cite{bettin-chandee} gives a somewhat superior bound (albeit using a similar method). Their bound (\cite[Theorem 1]{bettin-chandee}, taking in their notation $A = 1$, $\theta = b$, $M = pQ$, $N = p' Q$) gives, recalling that $p, p' \leq Z$, the bound
\[ \Sigma \ll Q^{2 - \frac{1}{20} + o(1)} Z.\] It can be checked that this is indeed bounded by $X^{-2\eta - \sigma/2}$ (as required by \eqref{to-show-47}) provided that $\eta \leq \frac{1}{78} - \sigma$.

\section{On allowing the residue class to vary}\label{var-cong-sec}

Suppose that $Q > CX^{1/2}$. Our results required a fixed residue class $a\md{q}$. If the residue class is allowed to depend on $q$, the problem appears to be vastly more difficult. Let us imagine taking a similar approach. Then, even in the case $\xi_q = 1$ for $q$ prime in Proposition \ref{prop3.1}, one would be led to bilinear forms of the type
\begin{equation}\label{eq111} \sum_{Q \leq q, q' < 2Q}\sum_{m \leq X} ( 1_{pm \equiv a(q) \md{q}} 1_{p' m \equiv a(q') \mdsub{q'}} - \frac{1}{qq'}),\end{equation}
and one would be seeking a bound of $o(X)$. 
Now suppose that $a(q) = p$ for $q \in S$, and that $a(q') = p'$ for $q' \in S'$, where $S, S'$ are disjoint sets, each consisting of half the primes in $[Q,2Q]$. Then one may check that \eqref{eq111} is $\gg Q^2$, the point being that if $q \in S$ and $q' \in S'$ then the unique solution $\md{qq'}$ to $pm \equiv a(q) \md{q}$ and $p' m \equiv a(q') \md{q'}$ is $ m = 1$, which automatically lies in $\{1,\dots, X\}$.

Thus to make progress, even in the special case $\xi_q = 1$, one would need a different mode of argument exploiting some averaging in $p, p'$, perhaps.

The following, say for $Q = X^{1/2 + \delta}$ for very small $\delta$, is an easier problem to which we do not know the solution. Suppose that for each prime $q \in [Q,2Q]$ we take a residue class $a(q) \md{q}$. Let $A$ be the union of all these residue classes, intersected with $\{1,\dots,X\}$. Is $\# A \gg X^{1 - o(1)}$? The connection between this problem and the distribution of multiplicative functions on progressions is discussed in \cite{green-ruzsa}. It is somewhat reminiscent of the Kakeya problem in Euclidean harmonic analysis and indeed implies it as $\delta \rightarrow \frac{1}{2}$. For more details see \cite{green-ruzsa}.

\end{document}